\documentclass[11pt]{amsart}

\usepackage[T1]{fontenc}
\usepackage{lmodern}
\usepackage{amsmath,amssymb,amsthm,mathtools}
\usepackage{enumitem}
\usepackage{microtype}
\usepackage[colorlinks=true,linkcolor=blue,citecolor=blue,urlcolor=blue]{hyperref}
\newtheorem{theorem}{Theorem}[section]

\newtheorem{corollary}[theorem]{Corollary}
\newtheorem{lemma}[theorem]{Lemma}
\newtheorem{proposition}[theorem]{Proposition}
\newtheorem{thmA}{Theorem}

\theoremstyle{definition}

\newtheorem{question}[theorem]{Question}
\newtheorem{definition}[theorem]{Definition}
\numberwithin{equation}{section}

\theoremstyle{remark}
\newtheorem{remark}[theorem]{Remark}
\newtheorem{example}[theorem]{Example}

\newcommand{\Z}{\mathbb Z}

\newcommand{\Gamp}{\mathop{\amalg}}
\newcommand{\Aut}{\operatorname{Aut}}

\newcommand{\GL}{\operatorname{GL}}
\newcommand{\cl}[1]{\overline{#1}}
\newcommand{\dd}{\operatorname{d}}

\title[bounded generation]{Bounded generation of pro-$p$ completions of amalgamated products}
\author[D. Peng]{Dekui Peng}
 \address[D. Peng]
{Institute of Mathematics, Nanjing Normal University, Nanjing 210024, China}
\email{pengdk10@lzu.edu.cn}

\author[J. Yang]{Jiang Yang\textsuperscript{*}}
 \address[J. Yang]
{School of Mathematical Sciences, Guangxi Minzu University, Nanning 530006, China}
\email{yangjiangdy@126.com}
\thanks{*Corresponding author.}

\hypersetup{
  pdftitle={Bounded generation of pro-p completions of amalgamated products},
  pdfauthor={Dekui Peng and Jiang Yang},
  pdfsubject={Bounded generation of pro-p completions of amalgamated products},
  pdfkeywords={bounded generation, pro-p completion, profinite completion, amalgamated product}
}


\begin{document}

\begin{abstract}
We study bounded generation in pro-$p$ completions of abstract amalgamated
products. We first give a complete characterization of nondegenerate proper
amalgamated free pro-$p$ products: if $G=G_1\Gamp_A G_2$, then $G$ is
boundedly generated if and only if $A$ is boundedly generated and
$p=[G_1:A]=[G_2:A]=2$.

We then consider an arbitrary abstract amalgam $\Gamma=H_1*_K H_2$ with
boundedly generated vertex groups. Writing $V_i$ for the closures of the
vertex images in $\widehat\Gamma_p$ and $B=V_1\cap V_2$, we characterize
bounded generation of $\widehat\Gamma_p$ in terms of the residual indices
$[V_i:B]$. Precisely, $\widehat\Gamma_p$ is boundedly generated if and only
if one vertex image is contained in the other, or $p=2$ and both residual
indices are equal to $2$.

Finally, for every prime $p$ we construct finitely generated, linear,
residually $p$, and boundedly generated groups $K\leq H$ such that, for
$\Gamma=H*_K H$, the group $\Gamma$ is finitely generated and residually
finite but not residually $p$, while
$\widehat\Gamma_p\cong\widehat H_p$ is boundedly generated and the full
profinite completion $\widehat\Gamma$ is not. In this example $K$ is
pro-$p$ dense in $H$, but the natural map
$\widehat K_p\to\widehat H_p$ is surjective and noninjective. We also give
intrinsic criteria ensuring that an abstract amalgam enters the proper
pro-$p$ Bass--Serre framework.
\end{abstract}

\maketitle
\tableofcontents

\section{Introduction}

\subsection*{Bounded generation and the motivating questions}

An abstract group $G$ is said to be \emph{boundedly generated} or \emph{have bounded generation} if it is a
product of finitely many cyclic subgroups. Thus there are elements
$g_1,\ldots,g_d\in G$ such that
\[
G=\langle g_1\rangle\cdots\langle g_d\rangle.
\]
For a profinite group we use the analogous topological definition: $G$ is
boundedly generated if there are $g_1,\ldots,g_d\in G$ such that
$G=\overline{\langle g_1\rangle}\cdots
\overline{\langle g_d\rangle}$.
Equivalently, a boundedly generated profinite group is a finite product of
procyclic subgroups.

We point out a difference in terminology that will be relevant in
Section~\ref{sec:arithmetic-nonproper-collapse}. Morris uses the phrase
``$X$ boundedly generates $G$'' for a subset $X\subseteq G$ if there is
an integer $r$ such that every element of $G$ is represented by a word of
length at most $r$ in $X\cup X^{-1}$; see
\cite[Definition~1.1]{Morris2007}. This is a bounded-width condition
relative to a specified subset, and is not literally our definition of
bounded generation.

The two notions agree in the situation needed here. If
$X\subseteq C_1\cup\cdots\cup C_m$, where the $C_i$ are cyclic
subgroups, and $X$ boundedly generates $G$ in the sense of Morris, then
$G$ is boundedly generated in our sense. Indeed, a uniform bound on the
word length leaves only finitely many possible sequences of the cyclic
subgroups from which the factors are chosen. Conversely, if
$G=C_1\cdots C_m$, then $C_1\cup\cdots\cup C_m$ boundedly generates
$G$ in the sense of Morris.

For example, the elementary matrices in $\mathrm{SL}_n(\mathbb Z)$ are
contained in the finitely many root subgroups
$U_{ij}=\{I+aE_{ij}:a\in\mathbb Z\}=\langle I+E_{ij}\rangle$,
$i\neq j$. Hence bounded elementary generation in the sense of
\cite{Morris2007} implies bounded generation in the sense used in this
paper. This observation explains why bounded elementary generation results can be
used to obtain bounded generation in the sense of this paper. In the arithmetic
construction below we use the theorem of Carter and Keller
\cite{CarterKeller1983}.

Bounded generation has been studied mainly in the theory of arithmetic
groups. A basic result of Carter and Keller states that
$\mathrm{SL}_n(\mathcal O)$ is boundedly generated by elementary
matrices for $n\geq3$, where $\mathcal O$ is the ring of integers of a
number field \cite{CarterKeller1983}. Tavgen obtained analogous results
for broad classes of higher-rank Chevalley groups over rings of
$S$-integers \cite{Tavgen1991}. The rank-one case is more delicate; see,
for example, \cite{Morris2007,MorganRapinchukSury2018}. On the negative
side, recent work also shows that bounded generation places strong
restrictions on anisotropic linear groups; see \cite{CRRZ2022}. These
results illustrate the general principle that bounded generation is a
strong finiteness property and is closely related to higher-rank
phenomena.

The questions studied in this paper arise from free constructions.
Let
$\Gamma=H_1*_K H_2$
be an amalgamated free product of abstract groups. In
\cite[Questions~9.5.5 and~9.5.6]{RibesZalesskii2010}, Ribes and
Zalesskii record the following two questions, proposed by A.~S.~Rapinchuk:

\begin{question}\label{q1}
Give verifiable sufficient conditions ensuring that the profinite
completion $\widehat\Gamma$ is boundedly generated.
\end{question}

\begin{question}\label{q2}
Fix a prime $p$. Give verifiable sufficient conditions ensuring that the
pro-$p$ completion $\widehat\Gamma_p$ is boundedly generated.
\end{question}

The second question is our main concern. At a fixed prime there is additional
structure: by \cite[Theorem~3.17]{DDMS1999}, bounded generation of a pro-$p$
group is equivalent to finite Prüfer rank. Ribes and Zalesskii also emphasize
that some residual hypothesis, or another condition preventing collapse in the
completion, is necessary for the questions to have their intended meaning.

There is already a strong obstruction in the abstract category. Results
of Grigorchuk and Fujiwara, obtained by bounded-cohomological methods,
show that a sufficiently nontrivial amalgamated product is not boundedly
generated; see \cite{Grigorchuk1996,Fujiwara2000} and the discussion in
\cite[Section~9.5]{RibesZalesskii2010}. In particular, for a
nondegenerate amalgam, the presence of at least three double cosets
$K\backslash H_i/K$ on one side gives an obstruction to bounded
generation.

The completion problem is different. Passing from $\Gamma$ to
$\widehat\Gamma_p$ may identify elements of the vertex groups, enlarge
the intersection of their images, or even make one vertex image
contained in the other. Thus the original Bass--Serre decomposition of
$\Gamma$ need not survive unchanged in its pro-$p$ completion. This collapse gives one of the two positive mechanisms that occur in our
classification.

The motivation in \cite[Section~9.5]{RibesZalesskii2010} also comes from
arithmetic groups. Many $S$-arithmetic groups associated with quaternion
algebras admit descriptions by amalgamated free products, and a better
understanding of bounded generation of their finite or profinite
quotients is related to questions around the congruence subgroup
property. This gives a natural reason to study bounded generation
together with Bass--Serre decompositions and finite quotients.

\subsection*{Main results}

We first consider the situation in which no collapse occurs and the
pro-$p$ group itself is a proper amalgamated free pro-$p$ product. The
following result gives a complete answer in this case.

\begin{thmA}
\label{thm:intro-proper-classification}
Let $G=G_1\Gamp_A G_2$ be a proper pro-$p$ amalgamated free product,
and assume that $A\lneq G_i$ for $i=1,2$. Then $G$ is boundedly generated
if and only if $A$ is boundedly generated and
$p=[G_1:A]=[G_2:A]=2$.
\end{thmA}

This is proved as Theorem~\ref{thm:proper-prop-bg-classification}. Thus
among nondegenerate proper pro-$p$ amalgams there is only one possible
bounded-generation profile: the prime is $2$ and both vertex groups have
index $2$ over the edge group. The proof reduces the general proper case to pro-$p$ amalgams with finite
vertex groups, where the Bass--Serre tree gives the required obstruction.

Recent work of Kionke, Otmen, Toti, Vannacci and Weigel
\cite[Theorem~4.5 and Corollary~4.8]{KOTVW2026} gives a related
largeness obstruction for proper pro-$p$ amalgams outside the pro-$2$
double-index-two case. In particular, their results also imply
non-bounded generation in the negative cases of
Theorem~\ref{thm:intro-proper-classification}, since a large pro-$p$ group
has an open subgroup mapping onto a nonabelian free pro-$p$ group and
therefore cannot have finite Prüfer rank. Our proof gives a direct
finite-rank argument for this obstruction and identifies bounded generation
in the remaining double-index-two case. We then pass from proper pro-$p$
amalgams to the residual profiles of abstract amalgams.

For an abstract amalgam $\Gamma=H_1*_K H_2$, however, the intrinsic
completions $(\widehat{H_i})_p$ and $\widehat K_p$ need not form a proper
amalgam inside $\widehat\Gamma_p$. To describe what actually survives,
let $\rho:\Gamma\to P=\widehat\Gamma_p$ be the canonical map and put
$V_i=\overline{\rho(H_i)}$ and $B=V_1\cap V_2$. Thus $B$ is the
intersection of the two closed vertex images; in general it may be larger than
$\overline{\rho(K)}$.

Our main fixed-prime result is the following.

\begin{thmA}
\label{thm:intro-residual-answer}
Let $\Gamma=H_1*_K H_2$, fix a prime $p$, and put
$P=\widehat\Gamma_p$,
$V_i=\overline{\rho(H_i)}$ and $B=V_1\cap V_2$, where
$\rho:\Gamma\to P$ is the canonical map. Assume that $H_1$ and $H_2$
are boundedly generated. Then $P$ is boundedly generated if and only if
one of the following holds:
\begin{enumerate}[label=\textup{(\alph*)}]
\item $B=V_1$ or $B=V_2$;
\item $p=2$ and $[V_1:B]=[V_2:B]=2$.
\end{enumerate}
\end{thmA}

See Theorem~\ref{B}. The point is that $P$ itself has a canonical proper residual
decomposition
$P\cong V_1\Gamp_BV_2$.
Theorem~\ref{thm:intro-proper-classification} can therefore be applied
to the groups that actually occur inside the completion, rather than to
the original abstract vertex and edge groups.

The two alternatives in Theorem~\ref{thm:intro-residual-answer} have
different meanings. In case (a), one vertex image is contained in the
other and the amalgam folds in the pro-$p$ completion. In case (b), the
residual splitting remains nondegenerate, but it has the exceptional
pro-$2$ double-index-two form. Thus, within the class of boundedly
generated vertex groups, these are exactly the two ways in which a
boundedly generated pro-$p$ completion can occur.

In particular, if the original splitting enters the pro-$p$ category
without collapse, so that
$\widehat\Gamma_p\cong
\widehat{(H_1)}_p\Gamp_{\widehat K_p}\widehat{(H_2)}_p$
properly, then Theorem~\ref{thm:intro-residual-answer} reduces to
Theorem~\ref{thm:intro-proper-classification}. This gives a direct
sufficient condition relevant to Question~\ref{q2}: in the nondegenerate
case, if $p=2$, both completion-level indices are $2$, and
$\widehat K_2$ is boundedly generated, then $\widehat\Gamma_2$ is
boundedly generated.

It remains to understand whether the folding alternative is only a
formal possibility. Our third main result shows that it occurs in a
natural arithmetic setting.

\begin{thmA}
\label{thm:intro-arithmetic-collapse}
For every prime $p$ there exist finitely generated, linear, residually
$p$, and boundedly generated groups $K\leq H$ such that $[H:K]$ is
finite and coprime to $p$, and $K$ is dense in $H$ for the pro-$p$
topology, while the natural map
$\widehat K_p\to\widehat H_p$
is surjective but not injective.

For $\Gamma=H*_K H$, the group $\Gamma$ is finitely generated and
residually finite but not residually $p$,
$\widehat\Gamma_p\cong\widehat H_p$
is boundedly generated, whereas the full profinite completion
$\widehat\Gamma$ is not boundedly generated.
\end{thmA}

This is Theorem~\ref{C}. The example has two purposes. First, it shows that
the folding alternative in Theorem~\ref{thm:intro-residual-answer} is
genuine, even when the vertex and edge groups are finitely generated,
linear, and residually $p$. Moreover, the map from the intrinsic pro-$p$
completion of the edge group to that of the vertex group may itself have a
nontrivial kernel.

Second, the example separates the two completion properties at a prescribed
prime in the presence of pro-$p$ collapse: bounded generation of
$\widehat\Gamma_p$ does not imply bounded generation of the full profinite
completion. Since the resulting amalgam $\Gamma$ is not residually $p$, this
does not resolve the residually-$p$, non-collapsing form of
Question~\ref{q2}, nor Question~\ref{q1}.

For completeness, we also study when the original splitting enters the
proper pro-$p$ category without collapse. The basic condition is
$p$-embeddedness: an inclusion $S\leq R$ is $p$-embedded when the
pro-$p$ topology of $R$ induces the full pro-$p$ topology on $S$.
We prove that
$\widehat\Gamma_p\cong
\widehat{(H_1)}_p\Gamp_{\widehat K_p}\widehat{(H_2)}_p$
properly if and only if $K$ is $p$-embedded in each $H_i$ and each
$H_i$ is $p$-embedded in $\Gamma$. We then record useful sufficient
conditions, including $p$-virtual retracts, finitely generated
nilpotent vertex groups whose inclusions in $\Gamma$ are $p$-embedded
(for instance, $p$-virtual retracts), and a finite normal-edge criterion
formulated in terms of the joint conjugation action. These results concern the
stronger intrinsic entrance problem; they are not needed for the
residual classification of Theorem~\ref{thm:intro-residual-answer}.

\subsection*{Notation and Bass--Serre terminology}

We finish the introduction by fixing the notation used throughout the
paper.

For an abstract group $H$, we write $\widehat H$ for its profinite
completion and $\widehat H_p$ for its pro-$p$ completion. A group is
\emph{residually finite} if the canonical map $H\to\widehat H$ is
injective, and it is \emph{residually $p$} if
$H\to\widehat H_p$ is injective.

For $x$ in an abstract group, $\langle x\rangle$ denotes the cyclic
subgroup generated by $x$. If $x$ belongs to a profinite group, then
$\overline{\langle x\rangle}$ denotes the procyclic subgroup topologically
generated by $x$. More generally, $\overline S$ denotes the closure of a
subset $S$. We write $C_p$ for the cyclic group of order $p$, $\mathbb Z_p$ for the
additive group of $p$-adic integers, and $F_r$ for the abstract free group
of rank $r$. The symbol
$N\triangleleft_o G$ means that $N$ is an open normal subgroup of the
profinite group $G$. For a subgroup $H\leq G$,
$\operatorname{Core}_G(H)=\bigcap_{g\in G}gHg^{-1}$ denotes its normal
core. We write $\dd(G)$ for the least cardinality of a topological generating
set of a profinite group $G$.

Let $H_1$ and $H_2$ be abstract groups containing a common subgroup
$K$. We call $(H_1,K,H_2)$ an \emph{amalgam}, and write
$H_1*_K H_2$ for its free product with amalgamation. The groups $H_1$
and $H_2$ are called the \emph{vertex groups}, and $K$ is called the
\emph{edge group}. This terminology comes from Bass--Serre theory:
$H_1*_K H_2$ is the fundamental group of a graph of groups consisting
of two vertices joined by one edge.

The associated Bass--Serre tree $T$ has vertex set
$G/H_1\sqcup G/H_2$ and edge set $G/K$, where
$G=H_1*_K H_2$. The edge represented by $gK$ joins the vertices
$gH_1$ and $gH_2$. Hence the stabilizers of vertices are conjugates of
the vertex groups and the stabilizers of edges are conjugates of the
edge group. We use these elementary facts repeatedly; see
\cite{Serre1980} for the classical theory.

Similarly, for pro-$p$ groups $G_1$ and $G_2$ with a specified common
closed subgroup $A$, we call $(G_1,A,G_2)$ a \emph{pro-$p$ amalgam}.
We write
$G_1\Gamp_A G_2$
for its free pro-$p$ product with amalgamation, namely the pushout of
the diagram $G_1\leftarrow A\rightarrow G_2$ in the category of
pro-$p$ groups. The pro-$p$ amalgam $(G_1,A,G_2)$ is called
\emph{proper} if the canonical maps
$G_i\to G_1\Gamp_A G_2$ are injective for $i=1,2$.

If $(G_1,A,G_2)$ is proper and $G=G_1\Gamp_A G_2$, then the standard
pro-$p$ Bass--Serre tree has the same formal description: its two
types of vertices are represented by the coset spaces $G/G_1$ and
$G/G_2$, its edges by $G/A$, and its vertex and edge stabilizers are
conjugates of $G_1,G_2$ and $A$, respectively. We refer to
\cite{RZ2000,Ribes2017,HZZ} for pro-$p$ trees and profinite
Bass--Serre theory.

\begin{definition}
\label{def:exact-prop-entrance}
For an abstract splitting $\Gamma=H_1*_K H_2$, we say that the splitting
has \emph{exact pro-$p$ entrance} if the natural maps identify
$\widehat\Gamma_p$ with
$\widehat{(H_1)}_p\Gamp_{\widehat K_p}\widehat{(H_2)}_p$ as a proper
pro-$p$ amalgam.
\end{definition}
The later sections explain both how to verify this condition and why it is
not necessary for the residual classification in
Theorem~\ref{thm:intro-residual-answer}.

\section{Preliminaries on bounded generation}
\label{sec:preliminaries}

We start with two elementary lemmas that will be used throughout.

\begin{lemma}
\label{lem:extension-bg}
Let
\[
        1\longrightarrow N\longrightarrow G\longrightarrow Q\longrightarrow 1
\]
be an exact sequence of profinite groups.  If $N$ and $Q$ are boundedly generated, then $G$ is boundedly generated.
\end{lemma}

\begin{proof}
Write $Q=D_1\cdots D_n$
with each $D_i$ procyclic.  Choose elements $x_i\in G$ whose images topologically generate the factors $D_i$.  Then
\[
        G=\cl{\langle x_1\rangle}\cdots \cl{\langle x_n\rangle}N.
\]
Thus $G$ is boundedly generated, as $N$ is  boundedly generated.
\end{proof}

\begin{lemma}
\label{lem:completion-quotient-bg}
Let $\Lambda$ be a boundedly generated abstract group.  If
$\varphi:\Lambda\to G$ has dense image in a profinite group $G$, then
$G$ is boundedly generated.  In particular, every profinite or pro-$p$
completion of $\Lambda$ is boundedly generated.  Moreover, every continuous
quotient of a boundedly generated profinite group is boundedly generated.
\end{lemma}

\begin{proof}
Write $\Lambda=\langle x_1\rangle\cdots\langle x_r\rangle$.
The compact set $\overline{\langle\varphi(x_1)\rangle}\cdots \overline{\langle\varphi(x_r)\rangle}$
is closed in $G$ and contains the dense subgroup $\varphi(\Lambda)$; hence it
is all of $G$.  The assertion about continuous quotients follows by taking
the images of the procyclic factors.
\end{proof}

\begin{lemma}
\label{lem:uniform}
Let $G$ be a profinite group and let $x_1,\dots,x_r\in G$.  Then the following assertions are equivalent:

\begin{itemize}
\item[(1)] $G=\cl{\langle x_1\rangle}\cdots \cl{\langle x_r\rangle}$;
\item[(2)] for every open normal subgroup $N\triangleleft_oG$, one has
\[
        G/N=\langle x_1N\rangle\cdots \langle x_rN\rangle.
\]
\end{itemize}
\end{lemma}

\begin{proof} Only the implication $(2)\Longrightarrow (1)$ needs a proof.
The product
\[
        X=\cl{\langle x_1\rangle}\cdots \cl{\langle x_r\rangle}
\]
 is a compact and hence closed subset of $G$.  If its image under the quotient homomorphism is all of $G/N$ for every open normal subgroup $N$, then $X$ is dense in $G$. Consequently, $X=G$. 
\end{proof}

This lemma is the point at which the full profinite problem differs from the pro-$p$ problem.  For $\widehat\Gamma$ one must find the same finite list of elements working in every finite quotient.

\begin{lemma}
\label{lem:open-bg}
Let $G$ be a boundedly generated profinite group.  Then every open subgroup of $G$ is boundedly generated.
\end{lemma}

\begin{proof}
First assume that $U\triangleleft_o G$.  Write
\[
        G=C_1\cdots C_r,
        \qquad C_i=\cl{\langle x_i\rangle}.
\]
 Since $U$ is open and normal, $D_i=C_i\cap U$ is an open subgroup of $C_i$, hence procyclic.  Choose a finite set $R_i\subseteq C_i$ of representatives for the right cosets of $D_i$ in $C_i$.

For a tuple $\boldsymbol\rho=(\rho_1,\ldots,\rho_r)\in R_1\times\cdots\times R_r$, put
\[
        D_{1,\boldsymbol\rho}=D_1,
        \qquad
        D_{i,\boldsymbol\rho}=(\rho_1\cdots\rho_{i-1})D_i(\rho_1\cdots\rho_{i-1})^{-1}
        \quad (2\le i\le r).
\]
Every $D_{i,\boldsymbol\rho}$ is a procyclic subgroup of $U$.  Let $\mathcal R$ be the finite set of tuples for which $\rho_1\cdots\rho_r\in U$, and, for $\boldsymbol\rho\in\mathcal R$, put
\[
        E_{\boldsymbol\rho}=\cl{\langle \rho_1\cdots\rho_r\rangle}\leq U.
\]
If $u\in U$ and $u=c_1\cdots c_r$ with $c_i\in C_i$, write $c_i=d_i\rho_i$ with $d_i\in D_i$ and $\rho_i\in R_i$.  Then $\boldsymbol\rho=(\rho_1,\ldots,\rho_r)\in\mathcal R$ and
\[
\begin{aligned}
u
&= c_1\cdots c_r
 = d_1\rho_1\cdots d_r\rho_r \\
&= d_1
   \cdot (\rho_1 d_2\rho_1^{-1})
   \cdot (\rho_1\rho_2 d_3\rho_2^{-1}\rho_1^{-1})
   \cdots 
   (\rho_1\cdots\rho_{r-1}
   d_r
   \rho_{r-1}^{-1}\cdots\rho_1^{-1})
   \cdot (\rho_1\cdots\rho_r) \\
&\in D_{1, \boldsymbol{\rho}}\cdots D_{r,\boldsymbol{\rho}}E_{\boldsymbol{\rho}}.
\end{aligned}
\]
Thus $U$ is the union of finitely many sets
$B_{\boldsymbol\rho}=D_{1,\boldsymbol\rho}\cdots
D_{r,\boldsymbol\rho}E_{\boldsymbol\rho}$, each of which is a product of
procyclic subgroups contained in $U$. Enumerate these sets as
$B_1,\ldots,B_s$. Then $U=\bigcup_{j=1}^s B_j\subseteq B_1\cdots B_s$,
while the reverse inclusion holds because every factor occurring in every
$B_j$ is contained in the subgroup $U$. Hence $U=B_1\cdots B_s$ and $U$ is
boundedly generated.

For a general open subgroup $U$, let
\[
        N=\operatorname{Core}_G(U)=\bigcap_{g\in G}gUg^{-1}.
\]
Then $N\triangleleft_o G$ and $N\le U$.  By the normal case $N$ is boundedly generated, while $U/N$ is finite.  Lemma \ref{lem:extension-bg} implies that $U$ is boundedly generated.
\end{proof}

The preceding argument has an abstract analogue that will be needed later.
Indeed, the finite-index decomposition used in the proof of
Lemma~\ref{lem:open-bg} does not depend on compactness, except for replacing
cyclic subgroups by their closures. We record the abstract version
separately.

\begin{lemma}
\label{lem:finite-index-bg}
Let $G$ be a boundedly generated abstract group. Then every subgroup
of finite index in $G$ is boundedly generated.
\end{lemma}

\begin{proof}
Let $H\leq G$ have finite index and put
$N=\operatorname{Core}_G(H)$. Then $N$ is normal and of finite index
in $G$. The proof of Lemma~\ref{lem:open-bg}, with cyclic subgroups
in place of procyclic subgroups, shows that $N$ is boundedly generated.
Since $H/N$ is finite, choose representatives $t_1,\ldots,t_s\in H$ for
its elements. Then
$H=N\langle t_1\rangle\cdots\langle t_s\rangle$, so $H$ is boundedly
generated.
\end{proof}

Let us first note that, in general, a closed subgroup of a boundedly
generated profinite group need not itself be boundedly generated. In the
category of pro-$p$ groups, however, this phenomenon cannot occur.

Recall that the \emph{Prüfer rank} of a pro-$p$ group $G$ is
$\sup\{\dd(H):H\leq G\text{ is closed}\}$. By
\cite[Theorem~3.17]{DDMS1999}, a pro-$p$ group is boundedly generated if
and only if it has finite Prüfer rank. We therefore obtain the following.

\begin{theorem}
\label{thm:closed-subgroups-bg-prop}
Let $G$ be a boundedly generated pro-$p$ group. Then every closed subgroup
of $G$ is boundedly generated.
\end{theorem}

\begin{proof}
By \cite[Theorem~3.17]{DDMS1999}, the group $G$ has finite Prüfer rank.
Every closed subgroup $H\leq G$ has Prüfer rank at most that of $G$, and
hence again has finite Prüfer rank. Applying the same theorem to $H$ shows
that $H$ is boundedly generated.
\end{proof}

This yields the following standard consequence.

\begin{corollary}
\label{lem:free-prop-not-bg}
A free pro-$p$ group of rank at least $2$ is not boundedly generated.
\end{corollary}

\begin{proof}
If $F$ has infinite rank, then it does not have finite Prüfer rank and hence
is not boundedly generated by \cite[Theorem~3.17]{DDMS1999}. Suppose that
$F$ has finite rank $r\geq2$. Choose an epimorphism $F\twoheadrightarrow
\mathbb Z_p$ and, for $m\geq1$, let $U_m$ be the inverse image of
$p^m\mathbb Z_p$. Then $[F:U_m]=p^m$. By the Schreier formula for open
subgroups of free pro-$p$ groups; see
\cite[Corollary~4.4]{Lubotzky1982},
$\dd(U_m)=1+p^m(r-1)$. Thus the numbers of generators of closed subgroups
of $F$ are unbounded, so $F$ has infinite Prüfer rank. Again
\cite[Theorem~3.17]{DDMS1999} implies that $F$ is not boundedly generated.
\end{proof}

\section{Bounded generation in proper \texorpdfstring{pro-$p$}{pro-p} amalgams}
\label{sec:thm1}


We begin with finite vertex groups. This finite model contains the basic
obstruction used in the general proper case.
\begin{theorem}
\label{thm:FPA}
Let $(P_1,A,P_2)$ be a proper amalgam of finite $p$-groups, and put
$Q=P_1\amalg_A P_2$. Then $Q$ is boundedly generated if and only if
one of the following alternatives holds:
\begin{enumerate}[label=\textup{(\alph*)}]
\item $A=P_1$ or $A=P_2$;
\item $p=[P_1:A]=[P_2:A]=2$.
\end{enumerate}
\end{theorem}

\begin{proof}
If $A=P_1$, then
\[
        Q=P_1\Gamp_A P_2\cong P_2,
\]
which is finite, hence boundedly generated.  The case $A=P_2$ is identical.

Assume now that $A\lneq P_1$ and $A\lneq P_2$.  Put
\[
        m_i=[P_i:A]\qquad (i=1,2).
\]
Since $P_i$ is a finite $p$-group and $A$ is a proper subgroup, $m_i\ge p$.

Suppose first that $p=m_1=m_2=2$.
Then $A$ has index $2$ in each $P_i$, hence $A\triangleleft P_i$.  Since $Q$ is topologically generated by the images of $P_1$ and $P_2$, it follows that $A\triangleleft Q$.  Therefore
\[
        Q/A\cong (P_1/A)\Gamp (P_2/A)\cong C_2\Gamp C_2.
\]
The latter group is the infinite dihedral pro-$2$ group, isomorphic to $\Z_2\rtimes C_2$.  Hence $Q/A$ is boundedly generated.  Since $A$ is finite, $A$ is boundedly generated, and Lemma \ref{lem:extension-bg} implies that $Q$ is boundedly generated.

It remains to prove non-bounded generation in all other nondegenerate cases.  
Let $T$ be the standard pro-$p$ Bass--Serre tree of $Q$.
The group acts by left translation on $T$. Then, the stabilizer of each point in $Q/P_i$ is conjugate to $P_i$ and the stabilizer of each edge is conjugate to $A$.

Because $P_1$ and $P_2$ are finite subgroups of the Hausdorff pro-$p$ group $Q$, there exists an open normal subgroup $N$ of $Q$
such that
\[
     N\cap P_1 = 1 = N\cap P_2.
\]
 Since $N$ is normal in $Q$, it also intersects trivially every conjugate of $P_1$ and $P_2$. Therefore $N$ acts freely on $T$, and hence is a free pro-$p$ group by \cite[Theorem~2.3(b)]{HZZ}.

Let $m=[Q:N]$.
The quotient graph $N\backslash T$ has $\frac{m}{|P_1|}+\frac{m}{|P_2|}$
vertices and $\frac{m}{|A|}$ edges.  By the rank formula for open free pro-$p$ subgroups of fundamental
pro-$p$ groups of finite graphs of finite $p$-groups
\cite[Remark~3.6]{HZZ}, we have

\[
\begin{aligned}
        \dd(N)
        &=1+\frac{m}{|A|}-\frac{m}{|P_1|}-\frac{m}{|P_2|}  \\
        &=1+\frac{m}{|A|}\left(1-\frac1{m_1}-\frac1{m_2}\right).
\end{aligned}
\]
In every remaining nondegenerate case the coefficient in parentheses is positive. So the rank of $N$ is at least $2$.
By Corollary \ref{lem:free-prop-not-bg}, such a subgroup is not boundedly generated.  Since bounded generation passes to open subgroups by Lemma \ref{lem:open-bg}, $Q$ itself cannot be boundedly generated.  This proves the theorem.
\end{proof}

\begin{remark}
The negative part of Theorem \ref{thm:FPA} is the crucial point.  Except in the finite degeneracies and the pro-$2$ dihedral quotient, the finite $p$-amalgam quotient has an open free pro-$p$ subgroup of rank at least two.  Thus these cases give a genuine obstruction to bounded generation.
\end{remark}


The finite model above treats finite vertex images.  We now prove that, for an arbitrary proper pro-$p$ amalgam, finite quotients detect every non-dihedral nondegenerate case.  This is the point at which the infinite-index branch becomes tractable under the properness hypothesis.

We shall use the following standard intersection fact for proper profinite amalgams.  If
\[
        G=G_1\Gamp_A G_2
\]
is proper, then the two vertex groups meet in the edge group:
\[
        G_1\cap G_2=A.
\]
Indeed, in the standard pro-$p$ Bass--Serre tree the two fundamental vertex
stabilizers are $G_1$ and $G_2$, while the edge joining them has stabilizer
$A$. An element of $G_1\cap G_2$ fixes both endpoints and hence the edge, so
$G_1\cap G_2=A$.

\begin{lemma}
\label{lem:finite-quotient-visibility}
Let $G=G_1\Gamp_A G_2$ be a proper pro-$p$ amalgamated free product. If
$[G_1:A]\geq n$ for some integer $n\geq2$, then there exists an open
normal subgroup $N\triangleleft_oG$ such that, in the finite $p$-quotient
$E=G/N$, with $P_k=G_kN/N$ for $k=1,2$ and $A_N=P_1\cap P_2$, one has
$[P_1:A_N]\geq n$.
\end{lemma}

\begin{proof}
Choose $x_1,\ldots,x_n\in G_1$ representing distinct left cosets of $A$.
For $r\neq s$, we have $x_s^{-1}x_r\notin A=G_1\cap G_2$, and hence
$x_s^{-1}x_r\notin G_2$. Since $G_2$ is closed in $G$, there is an open
normal subgroup $N_{rs}\triangleleft_oG$ such that
$x_s^{-1}x_r\notin G_2N_{rs}$. Put
$N=\bigcap_{r\neq s}N_{rs}$.

We claim that the images of $x_1,\ldots,x_n$ in $P_1=G_1N/N$ represent
distinct left cosets modulo $A_N=P_1\cap P_2$. Otherwise, for some
$r\neq s$,
$(x_sN)^{-1}(x_rN)\in A_N\leq P_2$. Hence
$x_s^{-1}x_rN\in P_2=G_2N/N$, so $x_s^{-1}x_r\in G_2N$, contradicting
$N\leq N_{rs}$. Thus $[P_1:A_N]\geq n$.
\end{proof}

If $[G_1:A]$ is finite, taking $n=[G_1:A]$ in the lemma gives an
open normal subgroup $N\triangleleft_oG$ such that $[P_1:A_N]=n$.
Equivalently, $[G_1:G_1\cap G_2N]=n$. Since
$A\leq G_1\cap G_2N$ and $[G_1:A]=n$, it follows that
$G_1\cap G_2N=A$.

We can now state the proper pro-$p$ obstruction theorem.  It is the main bounded-generation classification at the level of proper pro-$p$ amalgams.

\begin{proposition}
\label{prop:Aug3}
Let $G=G_1\Gamp_A G_2$ be a proper pro-$p$ amalgamated free product, where
$A\neq G_i$ for $i=1,2$. If $G$ is boundedly generated, then $A$ is
boundedly generated and $p=[G_1:A]=[G_2:A]=2$.
\end{proposition}

\begin{proof}
The subgroup $A$ is boundedly generated by
Theorem~\ref{thm:closed-subgroups-bg-prop}.

We show that $[G_1:A]=2$. Suppose, towards a contradiction, that
$[G_1:A]\geq3$. Apply Lemma~\ref{lem:finite-quotient-visibility} to the
first vertex with $n=3$, and obtain $N_1\triangleleft_oG$. Since
$A\neq G_2$, apply the same lemma with the two vertex groups interchanged
and $n=2$, obtaining $N_2\triangleleft_oG$. Put
$N=N_1\cap N_2$, and set
$E=G/N$, $P_i=G_iN/N$, and $A_N=P_1\cap P_2$.

The coset separations obtained modulo $N_1$ and $N_2$ remain valid modulo
the finer subgroup $N$. Hence
$[P_1:A_N]\geq3$ and $[P_2:A_N]\geq2$. In particular, $A_N$ is proper in
both $P_1$ and $P_2$.

The pro-$p$ amalgamated free product $P_1\Gamp_{A_N}P_2$ is proper. Indeed, the inclusions
$P_i\hookrightarrow E$ agree on $A_N$ and therefore induce a continuous
homomorphism $P_1\Gamp_{A_N}P_2\to E$. Its restriction to each canonical
vertex group is the given inclusion $P_i\hookrightarrow E$, so the
canonical maps $P_i\to P_1\Gamp_{A_N}P_2$ are injective.

The quotient maps $G_i\to P_i$, followed by the canonical vertex maps,
agree on $A$. Hence the universal property gives a continuous epimorphism
$G\twoheadrightarrow P_1\Gamp_{A_N}P_2$. By Theorem~\ref{thm:FPA}, the
target is not boundedly generated. This contradicts the fact that bounded
generation passes to continuous quotients. Therefore $[G_1:A]=2$.

By symmetry, $[G_2:A]=2$. Finally, a finite index of a closed subgroup in
a pro-$p$ group is a power of $p$, so $p=2$.
\end{proof}

\begin{proposition}
\label{prop:dihedral-positive}
Let $G=G_1\Gamp_A G_2$ be a proper pro-$2$ amalgamated free product such that
$[G_1:A]=[G_2:A]=2$. Then $A\trianglelefteq G$ and
$G/A\cong \Z_2\rtimes C_2$,
where the nontrivial element of $C_2$ acts on $\Z_2$ by inversion.
Consequently, if $A$ is boundedly generated, then so is $G$.
\end{proposition}

\begin{proof}
Since $A$ has index $2$ in each $G_i$, it is normal in both $G_1$ and
$G_2$. As $G$ is topologically generated by $G_1$ and $G_2$, it follows
that $A\trianglelefteq G$.

The quotient maps $G_i\to G_i/A$ identify $G/A$ with the free pro-$2$
product \[
(G_1/A)\amalg(G_2/A)\cong C_2\amalg C_2.
\]
The latter is the infinite pro-$2$ dihedral group, and hence $C_2\amalg C_2\cong \Z_2\rtimes C_2$,
with $C_2$ acting on $\Z_2$ by inversion.

The final assertion follows directly from Lemma \ref{lem:extension-bg}.
\end{proof}

Combining Propositions \ref{prop:Aug3} and \ref{prop:dihedral-positive}, we have the following final characterization.

\begin{theorem}
\label{thm:proper-prop-bg-classification}
Let $G=G_1\Gamp_A G_2$
be a proper  pro-$p$ amalgamated free product and assume that $A$ is properly contained in both $G_1$ and $G_2$.
Then $G$ is boundedly generated if and only if $A$ is boundedly generated and
$p=[G_1:A]=[G_2:A]=2$.
\end{theorem}

\begin{theorem}
\label{thm:exact-prop-entrance-edge-classification}
Let
$\Gamma=H_1*_K H_2$
be an abstract amalgamated product, and fix a prime $p$.  Assume that the splitting has exact pro-$p$ entrance and put
$ G_i=\widehat{(H_i)}_p$ and $A=\widehat K_p$.
Assume $A$ is properly contained in $G_i$ for $i=1,2$.
Then
$\widehat\Gamma_p\text{ is boundedly generated}$
if and only if $A$ is boundedly generated and
\[
        p = [G_1:A]=[G_2:A]=2.
\]
\end{theorem}

\begin{proof}
This is a corollary of Theorem \ref{thm:proper-prop-bg-classification}.
\end{proof}

\section{Residual profiles of \texorpdfstring{pro-$p$}{pro-p} completions}
\label{sec:canonical-residual-properization}

Let $\Gamma=H_1*_K H_2$ be an abstract amalgamated free product, let
$P=\widehat\Gamma_p$, and let $\rho:\Gamma\to P$ be the canonical map. Put
$V_i=\overline{\rho(H_i)}$ for $i=1,2$ and $B=V_1\cap V_2$. The point of
using $B$, rather than $\overline{\rho(K)}$, is that the latter may be
properly contained in $V_1\cap V_2$.

\begin{theorem}
\label{B}
The natural inclusions induce an isomorphism
\[
        P\cong V_1\amalg_B V_2,
\]
and the amalgam on the right is proper. If $H_1$ and $H_2$ are boundedly
generated, then $P$ is boundedly generated if and only if either
$B=V_1$ or $B=V_2$, or
$p=[V_1:B]=[V_2:B]=2$.
\end{theorem}

\begin{proof}
Let $T$ be a pro-$p$ group and let $f_i:V_i\to T$ be continuous
homomorphisms that agree on $B$. The restrictions of $f_i\circ\rho$ to
$H_1$ and $H_2$ agree on $K$, so the universal property of
$H_1*_K H_2$ gives a homomorphism $f:\Gamma\to T$. This homomorphism is
continuous for the pro-$p$ topology on $\Gamma$. Indeed, if
$U\triangleleft_oT$, then
$\Gamma/f^{-1}(U)\cong f(\Gamma)U/U\leq T/U$, which is a finite
$p$-group. Hence $f$ extends uniquely to a continuous homomorphism
$P\to T$. Its restrictions to $V_i$ are $f_i$ by density. Thus $P$ has
the universal property of $V_1\amalg_BV_2$. Since the vertex maps are the
inclusions $V_i\hookrightarrow P$, this amalgam is proper.

Assume now that $H_1$ and $H_2$ are boundedly generated. By
Lemma~\ref{lem:completion-quotient-bg}, each $V_i$ is boundedly generated.
If $B=V_1$ or $B=V_2$, then $P$ is the other vertex group. Otherwise $B$
is proper in both $V_i$ and is boundedly generated by
Theorem~\ref{thm:closed-subgroups-bg-prop}. The conclusion follows from
Theorem~\ref{thm:proper-prop-bg-classification}.
\end{proof}

In particular, if $K$ is dense in $H_1$ for the pro-$p$ topology, then
$V_1\leq V_2$, so $B=V_1$ and $P=V_2$. Hence $P$ is a quotient of
$\widehat{H_2}_p$ and is boundedly generated whenever $H_2$ is boundedly
generated. The same statement holds with the two vertices interchanged.

We now express the theorem in terms of finite $p$-quotients. Set
$e_i=[V_i:B]$, where $e_i=\infty$ if $B$ is not open in $V_i$. For a finite $p$-quotient
$\theta:\Gamma\twoheadrightarrow S$, put
$Q_i(\theta)=\theta(H_i)$,
$B_\theta=Q_1(\theta)\cap Q_2(\theta)$, and
$a_i(\theta)=[Q_i(\theta):B_\theta]$. We call $\theta$ \emph{bad} if
$a_1(\theta)>1$ and $a_2(\theta)>1$, except in the case
$a_1(\theta)=a_2(\theta)=2$ (this forces $p=2$). Notice that in general
$B_\theta$ need not equal $\theta(K)$.

For two finite $p$-quotients $\eta:\Gamma\twoheadrightarrow S$ and
$\theta:\Gamma\twoheadrightarrow T$, we say that $\eta$ \emph{dominates}
$\theta$ if there is a homomorphism $h:S\to T$ such that
$\theta=h\circ\eta$.

\begin{corollary}
\label{thm:finite-bad-envelope-criterion}
For $i=1,2$,
\[
        e_i=\sup_{\theta:\Gamma\twoheadrightarrow S} a_i(\theta),
\]
where $\theta$ ranges over all finite $p$-quotients. If $H_1$ and
$H_2$ are boundedly generated, then $\widehat\Gamma_p$ is boundedly generated
if and only if $\Gamma$ has no bad finite $p$-quotient.
\end{corollary}

\begin{proof}
Every finite $p$-quotient $\theta$ factors through an epimorphism
$\pi:P\twoheadrightarrow S$. Then $Q_i(\theta)=\pi(V_i)$ and
$B_\theta=\pi(V_1)\cap\pi(V_2)$. The natural map from the left cosets of
$B$ in $V_i$ onto the left cosets of $B_\theta$ in $\pi(V_i)$ is
surjective, so $a_i(\theta)\leq e_i$.

Conversely, suppose $e_i\geq n$ for some $n\geq2$, and choose
$x_1,\ldots,x_n\in V_i$ in distinct left cosets of $B$. If
$\{i,j\}=\{1,2\}$, then $x_s^{-1}x_r\notin V_j$ for $r\neq s$.
Since $V_j$ is closed, there is an open normal subgroup
$N_{rs}\triangleleft_oP$ such that
$x_s^{-1}x_r\notin V_jN_{rs}$. With
$N=\bigcap_{r\neq s}N_{rs}$, the images of $x_1,\ldots,x_n$ in $P/N$
lie in distinct cosets modulo the intersection of the two vertex images.
Hence the quotient $\Gamma\to P/N$ has $a_i\geq n$. This proves the
formula for $e_i$.

If a finite $p$-quotient $\eta$ dominates $\theta$, then
$a_i(\eta)\geq a_i(\theta)$ for $i=1,2$. Indeed, the factor map sends the
$i$th vertex image onto the corresponding vertex image and sends the
intersection of the two vertex images into the corresponding intersection,
so it induces a surjection of coset spaces.

Assume that $H_1$ and $H_2$ are boundedly generated. If
$\widehat\Gamma_p$ is boundedly generated, Theorem~\ref{B} shows that
either one of the $e_i$ is $1$, or $p=2$ and $e_1=e_2=2$. Since
$a_i(\theta)\leq e_i$, no bad finite $p$-quotient exists.

Conversely, suppose that $\widehat\Gamma_p$ is not boundedly generated.
Then $e_1,e_2>1$, and, after interchanging the vertices if necessary,
$e_1\geq3$. Choose finite $p$-quotients
$\theta_1:\Gamma\twoheadrightarrow S_1$ and
$\theta_2:\Gamma\twoheadrightarrow S_2$ such that
$a_1(\theta_1)\geq3$ and $a_2(\theta_2)\geq2$. Let $S$ be the image of
the diagonal homomorphism
$(\theta_1,\theta_2):\Gamma\to S_1\times S_2$, and let
$\eta:\Gamma\twoheadrightarrow S$ be the induced epimorphism. The two
coordinate projections show that $\eta$ dominates both $\theta_1$ and
$\theta_2$. Hence $a_1(\eta)\geq3$ and $a_2(\eta)\geq2$, so $\eta$ is a
bad finite $p$-quotient.
\end{proof}

The same criterion may be checked on any cofinal family of finite
$p$-quotients: domination cannot decrease either $a_i$.

\section{An arithmetic dense-folding example}
\label{sec:arithmetic-nonproper-collapse}

The proper case was settled in Section~\ref{sec:thm1}. By
Theorem~\ref{thm:intro-residual-answer}, the other positive residual
profile is dense folding. We now show that this case genuinely occurs.
More precisely, for every prime $p$ we construct an amalgam whose
pro-$p$ completion folds to a boundedly generated vertex completion,
even though the natural map on edge completions is not injective. The
same example has a non-boundedly generated full profinite completion.

\begin{theorem}
\label{C}
For every prime $p$ there exist finitely generated, linear, residually
$p$, and boundedly generated groups $K\leq H$ such that $[H:K]$ is
finite and coprime to $p$, and $K$ is dense in $H$ for the pro-$p$
topology, while the natural map
$\widehat K_p\to\widehat H_p$ is surjective but not injective.

For $\Gamma=H*_K H$, the group $\Gamma$ is finitely generated and
residually finite but not residually $p$,
$\widehat\Gamma_p\cong\widehat H_p$ is boundedly generated, whereas the
full profinite completion $\widehat\Gamma$ is not boundedly generated.
\end{theorem}

\begin{proof}
Choose
\[
n=
\begin{cases}
3,&p\neq3,\\
4,&p=3.
\end{cases}
\]
Thus $n\geq3$ and $p\nmid n$. Let $\ell$ be a prime divisor of
$2^p-1$. Since $2^p\equiv1\pmod\ell$, the multiplicative order of $2$
modulo $\ell$ divides $p$. As $p$ is prime,
it is equal to $p$. Therefore
$p\mid |\mathbb F_\ell^\times|=\ell-1$, so
$\ell\equiv1\pmod p$. In particular, $\ell\neq p$.

Let
$H=\Gamma_n(p)=\ker(\mathrm{SL}_n(\mathbb Z)\to
\mathrm{SL}_n(\mathbb F_p))$,
the principal congruence subgroup of level $p$ (see \cite{LS}), and put
$Q=\mathrm{SL}_n(\mathbb F_\ell)$. Let
\[
\pi_\ell:\mathrm{SL}_n(\mathbb Z)\longrightarrow
\mathrm{SL}_n(\mathbb F_\ell)
\]
be reduction of matrix entries modulo $\ell$, and let
$\phi=\pi_\ell|_H$.

We claim that $\phi:H\rightarrow Q$ is surjective. Recall that
$Q$ is generated by the elementary matrices
$e_{ij}(a)=I+aE_{ij}$, where $E_{ij}$ denotes the standard matrix unit;
see, for example, \cite[Example~1.6]{Morris2007}. Given
$a\in\mathbb F_\ell$, the Chinese remainder theorem gives
$b\in\mathbb Z$ such that $b\equiv0\pmod p$ and
$b\equiv a\pmod\ell$. Hence $e_{ij}(b)\in H$ and
$\phi(e_{ij}(b))=e_{ij}(a)$. Thus $\phi$ is onto.

Fix a standard basis $e_1, \ldots, e_n$ of $\mathbb F_\ell^n$ and 
let $P$ be the setwise stabilizer in $Q$ of
$\mathbb F_\ell e_1$.  
Set $K=\phi^{-1}(P)$. Since $Q$ acts transitively on the set of one-dimensional subspaces of
$\mathbb F_\ell^n$, and this set has cardinality $\frac{\ell^n-1}{\ell-1}$, one has
\[
[H:K]=[Q:P]=\frac{\ell^n-1}{\ell-1}
=1+\ell+\cdots+\ell^{n-1}.
\]
As $\ell\equiv1\pmod p$, we have $[H:K]\equiv n\pmod p$. By the choice
of $n$, the index $[H:K]$ is coprime to $p$.

It follows that $K$ is dense in $H$ for the pro-$p$ topology. Indeed,
if $\alpha:H\to S$ is a homomorphism to a finite $p$-group, then
$[\alpha(H):\alpha(K)]$ is both a power of $p$ and a divisor of
$[H:K]$. Thus $\alpha(H)=\alpha(K)$.

We next record the basic properties of $H$ and $K$. For $r\geq1$, let
$\Gamma_n(p^r)$ denote the principal congruence subgroup of level
$p^r$. These subgroups have trivial intersection, and the map
$I+p^rX\mapsto X\pmod p$ induces an injection
\[
\Gamma_n(p^r)/\Gamma_n(p^{r+1})
\longrightarrow M_n(\mathbb F_p),
\]
where $M_n(\mathbb F_p)$ is regarded as an additive group. Hence
$H/\Gamma_n(p^m)$ is a finite $p$-group for every $m$, and $H$ is
residually $p$. Therefore its subgroup $K$ is also residually $p$.

Both $H$ and $K$ have finite index in $\mathrm{SL}_n(\mathbb Z)$,
so they are finitely generated and linear. By
\cite{CarterKeller1983}, $\mathrm{SL}_n(\mathbb Z)$ is boundedly generated
by elementary matrices. For each $i\neq j$,
\[
U_{ij}=\{e_{ij}(a):a\in\mathbb Z\}
      =\langle e_{ij}(1)\rangle
\]
is cyclic, and there are only finitely many such subgroups. Thus
$\mathrm{SL}_n(\mathbb Z)$ is boundedly generated in the sense used
in this paper. Lemma~\ref{lem:finite-index-bg} now implies that both
$H$ and $K$ are boundedly generated.

We now show that the natural map
$\widehat K_p\to\widehat H_p$ is not injective. Each $g\in P$ preserves
the line $\mathbb F_\ell e_1$, so there is a unique
$\lambda(g)\in\mathbb F_\ell^\times$ such that
$ge_1=\lambda(g)e_1$. This defines a homomorphism
$\lambda:P\to\mathbb F_\ell^\times$. It is surjective, since
$\operatorname{diag}(a,a^{-1},1,\ldots,1)\in P$ has image $a$ for every
$a\in\mathbb F_\ell^\times$.

Since $p\mid\ell-1$, the cyclic group $\mathbb F_\ell^\times$ has a
quotient group $C_p$ of order $p$. Composing this quotient with $\lambda$ and the
epimorphism $\phi|_K:K\twoheadrightarrow P$ gives a nontrivial homomorphism
$\chi:K\to C_p$.

We claim that $\chi$ does not extend to a homomorphism $H\to C_p$.
Indeed, $\chi$ vanishes on $\ker\phi$. If
$\psi:H\to C_p$ extended $\chi$, then $\psi$ would also vanish on
$\ker\phi$, and hence factor through
$Q=H/\ker\phi$. But $Q=\mathrm{SL}_n(\mathbb F_\ell)$ is perfect, as $n\geq 3$.
Thus every homomorphism $Q\to C_p$ is trivial, a contradiction.

Since $K$ is pro-$p$ dense in $H$, the natural map
$\widehat K_p\to\widehat H_p$ is surjective. If it were injective, it
would be an isomorphism. The homomorphism $\chi:K\to C_p$ would then
extend continuously to $\widehat K_p$, and through this isomorphism
would give a homomorphism $\widehat H_p\to C_p$. Restricting it to $H$
would extend $\chi$, which is impossible. Therefore
$\widehat K_p\to\widehat H_p$ is noninjective.

Now put $\Gamma=H*_K H$. If $\alpha:\Gamma\to S$ is a homomorphism to a
finite $p$-group, its restrictions to the two copies of $H$ agree on
$K$. Since $K$ is pro-$p$ dense in $H$, these restrictions agree on
all of $H$. Thus every homomorphism from $\Gamma$ to a finite $p$-group
factors through the folding map $\Gamma\to H$. Consequently
$\widehat\Gamma_p\cong\widehat H_p$, and this group is boundedly
generated by Lemma~\ref{lem:completion-quotient-bg}.

The group $\Gamma$ is not residually $p$. Indeed, since $K<H$, choose
$h\in H\setminus K$ and let $h^{(1)},h^{(2)}$ denote its copies in the two
vertex groups. By the normal form theorem,
$h^{(1)}(h^{(2)})^{-1}\neq1$ in $\Gamma$, whereas every homomorphism from
$\Gamma$ to a finite $p$-group factors through the folding map and therefore
kills this element.

It remains to consider the full profinite completion. First we show
that $\Gamma$ is residually finite. Let
$N=\operatorname{Core}_H(K)=\bigcap_{h\in H}hKh^{-1}$, the largest
normal subgroup of $H$ contained in $K$. Since $K$ has finite index in
$H$, so does $N$. The two copies of $N$ are identified in the amalgam,
and $N$ is normal in both vertex groups; hence $N\trianglelefteq\Gamma$.
Moreover, since $N\trianglelefteq H$ and $N\leq K$, the universal
property of amalgamated free products gives
\[
\Gamma/N\cong (H/N)*_{K/N}(H/N).
\]
Since $N$ has finite index in $H$, both $H/N$ and $K/N$ are finite.
Thus $\Gamma/N$ is the fundamental group of a finite graph of finite
groups, and hence is virtually free by Bass--Serre theory; see
\cite[Chapter~II, Section~2.6]{Serre1980}. In particular,
$\Gamma/N$ is residually finite.

Let $1\neq\gamma\in\Gamma$. If $\gamma\notin N$, its image in
$\Gamma/N$ can be separated from the identity by a finite quotient.
If $\gamma\in N$, then $\gamma$ is a nontrivial element of the
residually finite group $H$. Choose a finite quotient
$\rho:H\to B$ with $\rho(\gamma)\neq1$. Applying the same map $\rho$ to
both copies of $H$ gives a homomorphism $\Gamma\to B$ separating
$\gamma$. Thus $\Gamma$ is residually finite.

Finally, applying $\phi:H\twoheadrightarrow Q$ to both vertex groups
gives an epimorphism $\Gamma\twoheadrightarrow L:=Q*_P Q$.
As above, since $Q$ and $P$ are finite, $L$ is virtually free. More
precisely, let $\varepsilon:L=Q*_P Q\longrightarrow Q$
be the folding homomorphism whose restriction to each copy of $Q$ is
the identity. By \cite[Lemma~1.3]{BDGM}, the
kernel $F=\ker\varepsilon$ is a free group with $\mathrm{rank}(F)=[Q:P]-1$.
Moreover, $F\trianglelefteq L$ and
$[L:F]=|Q|$, so $F$ has finite index in $L$. Since
\[
[Q:P]=1+\ell+\cdots+\ell^{n-1}>2,
\]
we have $\mathrm{rank}(F)\geq2$. Thus $F$ is a finite-index normal nonabelian
free subgroup of $L$.

The profinite topology of $L$ induces the full profinite topology on
the finite-index subgroup $F$; see
\cite[Lemma~3.1.4]{RibesZalesskii2010}. Hence the closure $\overline F$
of $F$ in $\widehat L$ is naturally isomorphic to $\widehat F$ and is
open. If $\widehat L$ were boundedly generated, then $\widehat F$ would
be boundedly generated by Lemma~\ref{lem:open-bg}. However,
$\widehat F$ maps onto its pro-$\ell$ completion $\widehat F_\ell$,
which is a free pro-$\ell$ group of rank $\mathrm{rank}(F)\geq2$ and therefore is not
boundedly generated by Corollary~\ref{lem:free-prop-not-bg}. This is a
contradiction. Hence $\widehat L$ is not boundedly generated. Since
$\Gamma\twoheadrightarrow L$ induces
$\widehat\Gamma\twoheadrightarrow\widehat L$, the full profinite
completion $\widehat\Gamma$ is not boundedly generated.
\end{proof}

\section{Intrinsic criteria for exact \texorpdfstring{pro-$p$}{pro-p} entrance}
\label{sec:properness-entry}

The residual classification of Section~\ref{sec:canonical-residual-properization}
does not require the intrinsic completions of the original vertex and edge
groups to form a proper pro-$p$ amalgam. Nevertheless, it is useful to have
criteria ensuring that this stronger property does hold. We collect here a
small number of such criteria.

\begin{definition}
Let $S\leq R$. We say that $S$ is \emph{$p$-embedded} in $R$ if the
pro-$p$ topology of $R$ induces the full pro-$p$ topology on $S$.
Equivalently, for every normal subgroup $U\triangleleft S$ of finite
$p$-power index there is a normal subgroup $N\triangleleft R$ of finite
$p$-power index such that $N\cap S\leq U$.
\end{definition}

By \cite[Lemma~3.2.6]{RibesZalesskii2010}, $S$ is $p$-embedded in $R$
if and only if the inclusion $S\hookrightarrow R$ induces an injection
$\widehat S_p\to\widehat R_p$.

\begin{theorem}
\label{thm:pembedded-properness}
Let $\Gamma=H_1*_K H_2$. Then the following are equivalent.
\begin{enumerate}[label=\textup{(\alph*)}]
\item The splitting has exact pro-$p$ entrance in the sense of
Definition~\ref{def:exact-prop-entrance}; equivalently, the natural maps
identify
$\widehat\Gamma_p
\cong
\widehat{(H_1)}_p\Gamp_{\widehat K_p}\widehat{(H_2)}_p$
as a proper pro-$p$ amalgamated free product.
\item The subgroup $K$ is $p$-embedded in both $H_1$ and $H_2$, and
$H_1,H_2$ are $p$-embedded in $\Gamma$.
\end{enumerate}
In this case, taking all closures in $\widehat\Gamma_p$, one has
$\overline{H_1}\cap\overline{H_2}=\overline K$.
\end{theorem}

\begin{proof}
Assume (b). Since $K$ is $p$-embedded in $H_i$, the natural maps
$\widehat K_p\to\widehat{(H_i)}_p$ are injective. Hence the pro-$p$
amalgamated free product
$P=\widehat{(H_1)}_p\Gamp_{\widehat K_p}\widehat{(H_2)}_p$
is defined with $\widehat K_p$ identified with a closed subgroup of
each vertex group.

The maps $H_i\to P$ agree on $K$ and therefore induce a homomorphism
$\Gamma\to P$. Since $P$ is pro-$p$, this extends to a continuous
homomorphism $\widehat\Gamma_p\to P$. Conversely, the inclusions
$H_i\hookrightarrow\Gamma$ induce maps
$\widehat{(H_i)}_p\to\widehat\Gamma_p$, and these agree on
$\widehat K_p$. The universal property of the pro-$p$ amalgam gives
a map $P\to\widehat\Gamma_p$. The two maps are inverse on the dense
images of the vertex groups, hence are inverse isomorphisms.
The assumed $p$-embeddedness of $H_i$ in $\Gamma$ gives properness.

Conversely, properness gives injectivity of
$\widehat K_p\to\widehat{(H_i)}_p$ and
$\widehat{(H_i)}_p\to\widehat\Gamma_p$. The equivalence with
$p$-embeddedness follows from
\cite[Lemma~3.2.6]{RibesZalesskii2010}.

Finally, in a proper pro-$p$ amalgam the two vertex stabilizers of the
fundamental edge in the Bass--Serre tree intersect precisely in the edge
stabilizer. Hence $\overline{H_1}\cap\overline{H_2}=\overline K$.
\end{proof}

We shall use three elementary facts about $p$-embeddedness. First, it is
transitive: if $T\leq S\leq R$, with $T$ $p$-embedded in $S$ and $S$
$p$-embedded in $R$, then $T$ is $p$-embedded in $R$. Indeed, given a
normal subgroup $U\triangleleft T$ of finite $p$-power index, first choose
$V\triangleleft S$ of finite $p$-power index with $V\cap T\leq U$, and
then choose $N\triangleleft R$ of finite $p$-power index with
$N\cap S\leq V$.

Second, every retract is $p$-embedded. If $\rho:R\to S$ is a retraction
and $U\triangleleft S$ has finite $p$-power index, then
$\rho^{-1}(U)\triangleleft R$ has finite $p$-power index and
$\rho^{-1}(U)\cap S=U$. Finally, if $V$ is open in the pro-$p$ topology
of $R$, then the pro-$p$ topology of $R$ induces the full pro-$p$ topology
on $V$; see \cite[Lemma~3.1.4]{RibesZalesskii2010}. Consequently, if
$H\leq V\leq\Gamma$, where $V$ is open in the pro-$p$ topology of
$\Gamma$ and $H$ is a retract of $V$, then $H$ is $p$-embedded in
$\Gamma$. We use the term \emph{$p$-virtual retract} of $\Gamma$ for
such a subgroup $H$.

For nilpotent vertex groups the edge condition is automatic.
By \cite[Theorem~3.2]{Morales2024}, every subgroup of a finitely
generated nilpotent group is $p$-embedded.

\begin{corollary}
\label{cor:nilpotent-entrance}
Let $\Gamma=H_1*_K H_2$, where $H_1,H_2$ are finitely generated
nilpotent groups. If both vertex groups are $p$-embedded in $\Gamma$
---in particular, if they are $p$-virtual retracts of $\Gamma$---then
$\widehat\Gamma_p\cong
\widehat{(H_1)}_p\Gamp_{\widehat K_p}\widehat{(H_2)}_p$
as a proper pro-$p$ amalgamated free product. Hence the bounded-generation classification of
Theorem~\ref{thm:exact-prop-entrance-edge-classification} applies.
In the exceptional pro-$2$ double-index-two case,
$\widehat K_2$ is automatically boundedly generated.
\end{corollary}

\begin{proof} By \cite[Theorem~3.2]{Morales2024}, every subgroup of a finitely
generated nilpotent group is $p$-embedded.
So Theorem~\ref{thm:pembedded-properness} applies. Since
subgroups of finitely generated nilpotent groups are finitely generated
nilpotent, $\widehat K_2$ is a finitely generated nilpotent pro-$2$
group and hence is boundedly generated.
\end{proof}

We next record the finite normal-edge criterion that will be useful
below. Let $(P_1,A,P_2)$ be an amalgam of finite $p$-groups. We say that it
has a \emph{strong finite $p$-envelope} if there are a finite $p$-group $E$
and embeddings $\iota_i:P_i\hookrightarrow E$ which agree on $A$ and
satisfy $\iota_1(P_1)\cap\iota_2(P_2)=\iota_1(A)$.

A \emph{central $p$-filtration} of a finite $p$-group $P$ is a finite chain
$P=P_0\geq P_1\geq\cdots\geq P_c=1$ such that
$[P,P_r]P_r^p\leq P_{r+1}$ for every $r<c$, where $P_r^p$ denotes the subgroup generated by the $p$th powers of
elements of $P_r$. Two central $p$-filtrations
\[
\mathcal P_i:\quad P_i=P_{i,0}\geq P_{i,1}\geq\cdots\geq P_{i,c_i}=1,
\qquad i=1,2,
\]
are \emph{compatible over $A$} if, after padding the shorter filtration
by terminal copies of $1$ when necessary, they have the same length, i.e., $c_1=c_2$, and
\[
A\cap P_{1,r}=A\cap P_{2,r}
\]
at every level $r$.

By \cite[Theorem~2.1]{AschenbrennerFriedl2013}, the existence of such
compatible central $p$-filtrations is equivalent to the existence of a strong finite
$p$-envelope. Higman's main theorem \cite[pp.~301--305]{Higman1964} identifies the
same finite-amalgam condition with residual $p$-finiteness of $P_1*_A P_2$.

\begin{theorem}
\label{thm:normal-edge-joint-action}
Let $(P_1,A,P_2)$ be an amalgam of finite $p$-groups with
$A\trianglelefteq P_i$ for $i=1,2$, and put $\Omega$ to be the subgroup of $\Aut(A)$ generated by the actions of $P_i$ on $A$ by conjugation.
Then $(P_1,A,P_2)$ has a strong finite $p$-envelope if and only if
$\Omega$ is a $p$-group.
\end{theorem}

\begin{proof}
Suppose first that the amalgam embeds strongly in a finite $p$-group
$E$, which we may assume is generated by the two vertex images.
Since $A$ is normal in both vertex groups, it is normal in $E$.
Therefore the image of the conjugation homomorphism
$E\to\operatorname{Aut}(A)$ is a finite $p$-group, and it contains
$\Omega$. Hence $\Omega$ is a $p$-group.

Conversely, suppose that $\Omega$ is a $p$-group and put
$S=A\rtimes\Omega$. Then $S$ is a finite $p$-group. Let
\[
S=S_0\geq S_1\geq\cdots\geq S_c=1
\]
be the lower exponent-$p$ central series of $S$, so that $S_{r+1}=[S,S_r]S_r^p$
for every $r<c$. Put $A_r=A\cap S_r$. Then
\[
A=A_0\geq A_1\geq\cdots\geq A_c=1.
\]
We claim that $[A_r,P_i]A_r^p\leq A_{r+1}$ for $i=1,2,\ r<c$.
Indeed, the conjugation action of each $P_i$ on $A$ factors through
its image in $\Omega$. Hence, for $x\in A_r$ and $g\in P_i$, if
$\omega\in\Omega$ denotes the automorphism of $A$ induced by $g$, then $[x,g]=[x,\omega]$.
Since $A_r\leq S_r$ and $\Omega\leq S$, it follows that
\[
[A_r,P_i]\leq A\cap[S_r,S]
\leq A\cap S_{r+1}
=A_{r+1}.
\]
Also,
\[
A_r^p\leq A\cap S_r^p
\leq A\cap S_{r+1}
=A_{r+1}.
\]
Therefore $[A_r,P_i]A_r^p\leq A_{r+1}$, as required.

For each $i=1,2$, choose a central $p$-filtration
\[
P_i/A=\overline P_{i,0}\geq \overline P_{i,1}\geq\cdots
\geq \overline P_{i,d_i}=1.
\]
After padding the shorter filtration by terminal copies of $1$, we may
assume that $d_1=d_2=d$. Let $P_{i,r}$ be the full preimage of
$\overline P_{i,r}$ under the quotient map $P_i\to P_i/A$. Then
\[
P_i=P_{i,0}\geq P_{i,1}\geq\cdots\geq P_{i,d}=A.
\]
Since the filtration of $P_i/A$ is central, we have
$[P_i,P_{i,r}]P_{i,r}^p\leq P_{i,r+1}$ for every $r<d$. Moreover,
each $P_{i,r}$ contains $A$, so
$A\cap P_{1,r}=A=A\cap P_{2,r}$ for $0\leq r\leq d$.

Now append to both chains the common filtration
$A=A_0\geq A_1\geq\cdots\geq A_c=1$ constructed above. Since
$[A_r,P_i]A_r^p\leq A_{r+1}$ for $i=1,2$ and $r<c$, the resulting
chains are central $p$-filtrations of $P_1$ and $P_2$. They are
compatible over $A$: before reaching $A$, both chains intersect $A$
in $A$, while from that point onward they coincide with the same
filtration $A=A_0\geq\cdots\geq A_c=1$. The Aschenbrenner--Friedl
criterion therefore yields a strong finite $p$-envelope.
\end{proof}

The joint condition cannot be checked separately on the two vertex
actions.

\begin{example}
\label{ex:GL2-joint-action-obstruction}
Let $A=C_2\times C_2$ and let
\[
U=\begin{pmatrix}
1&1\\
0&1
\end{pmatrix},
\qquad
V=\begin{pmatrix}
1&0\\
1&1
\end{pmatrix}
\]
in $\GL_2(\mathbb F_2)$. Put
$P_1=A\rtimes_U C_2$ and $P_2=A\rtimes_V C_2$.
Then $[P_i:A]=2$, but
$\langle U,V\rangle=\GL_2(\mathbb F_2)\cong S_3$,
which is not a $2$-group. Hence
$(P_1,A,P_2)$ has no strong finite $2$-envelope by
Theorem~\ref{thm:normal-edge-joint-action}. Thus the double-index-two
condition alone does not guarantee intrinsic pro-$2$ entrance.
\end{example}

The preceding finite criterion yields a useful intrinsic entrance
theorem for normal edge groups.

\begin{theorem}
\label{thm:normal-finite-p-extension-entry}
Let $\Gamma=H_1*_K H_2$, where $K\trianglelefteq H_i$ and $H_i/K$ is a
finite $p$-group for $i=1,2$. Assume that $K$ has a cofinal family
$\mathcal L$ of characteristic subgroups of finite $p$-power index, in the
sense that for every normal subgroup $U\triangleleft K$ of finite
$p$-power index there is $L\in\mathcal L$ with $L\leq U$, and assume
that, for every $L\in\mathcal L$, the subgroup of
$\operatorname{Aut}(K/L)$ generated by the conjugation actions of $H_1$
and $H_2$ is a $p$-group. Then
$\widehat\Gamma_p\cong
\widehat{(H_1)}_p\Gamp_{\widehat K_p}\widehat{(H_2)}_p$
as a proper pro-$p$ amalgam.
\end{theorem}

\begin{proof}
First, $K$ is $p$-embedded in each $H_i$. Given
$U\trianglelefteq K$ of finite $p$-power index, choose
$L\in\mathcal L$ with $L\leq U$. Since $L$ is characteristic in $K$
and $K\trianglelefteq H_i$, we have $L\trianglelefteq H_i$. Moreover,
$H_i/L$ is a finite $p$-group because both $K/L$ and $H_i/K$ are finite
$p$-groups. Thus the quotient $H_i\to H_i/L$ witnesses the required
$p$-embeddedness of $K$ in $H_i$.

It remains to show that each $H_i$ is $p$-embedded in $\Gamma$. Let
$U_i\trianglelefteq H_i$ have finite $p$-power index. Choose
$L\in\mathcal L$ such that $L\leq K\cap U_1\cap U_2$. The amalgamated free product 
$(H_1/L)\Gamp_{K/L} (H_2/L)$ has normal edge group, and by
hypothesis its two conjugation actions generate a $p$-group in
$\operatorname{Aut}(K/L)$. Theorem~\ref{thm:normal-edge-joint-action}
therefore gives a strong finite $p$-envelope $E$ of this finite amalgam.
The quotient maps $H_i\to H_i/L$, followed by the embeddings into $E$,
agree on $K$ and induce a homomorphism $\Gamma\to E$. Its kernel meets
each $H_i$ in exactly $L$, and hence in a subgroup of $U_i$. Thus each
$H_i$ is $p$-embedded in $\Gamma$.

Theorem~\ref{thm:pembedded-properness} now gives the conclusion.
\end{proof}

\begin{corollary}
\label{thm:normal-double-index-two-entry}
Assume the hypotheses of
Theorem~\ref{thm:normal-finite-p-extension-entry} with $p=2$ and
$[H_i:K]=2$. Then
\[
\widehat\Gamma_2
\cong
\widehat{(H_1)}_2\Gamp_{\widehat K_2}\widehat{(H_2)}_2
\]
as a proper pro-$2$ amalgam and
$[\widehat{(H_i)}_2:\widehat K_2]=2$ for $i=1,2$. Consequently,
\[
\widehat\Gamma_2\text{ is boundedly generated}
\quad\Longleftrightarrow\quad
\widehat K_2\text{ is boundedly generated}.
\]
In particular, the latter condition holds when $K$ is finitely
generated nilpotent and residually $2$.
\end{corollary}

\begin{proof}
Theorem~\ref{thm:normal-finite-p-extension-entry} gives the proper
pro-$2$ amalgam. For each $i$, the quotient map
$H_i\twoheadrightarrow H_i/K\cong C_2$ extends to a continuous
epimorphism $\widehat{(H_i)}_2\twoheadrightarrow C_2$. Since $K$ is
open in the pro-$2$ topology of $H_i$, its closure in
$\widehat{(H_i)}_2$ is exactly the kernel of this extension. Moreover,
$K$ is $2$-embedded in $H_i$, so this closure is naturally isomorphic to
$\widehat K_2$. Hence
$[\widehat{(H_i)}_2:\widehat K_2]=2$.

The bounded-generation equivalence now follows from
Theorem~\ref{thm:exact-prop-entrance-edge-classification}. If $K$ is
finitely generated nilpotent and residually $2$, then $\widehat K_2$ is
a finitely generated nilpotent pro-$2$ group and is therefore boundedly
generated.
\end{proof}

\begin{remark}
The preceding results concern intrinsic entrance at one fixed prime.
They are not needed for the residual-profile classification of
Section~\ref{sec:canonical-residual-properization}. Nor do they give
a criterion for bounded generation of the full profinite completion:
Theorem~\ref{thm:intro-arithmetic-collapse} shows that
$\widehat\Gamma_p$ may be boundedly generated for a prescribed prime
$p$ while $\widehat\Gamma$ is not. That example lies in the folding
branch and the corresponding amalgam is not residually $p$; it therefore
does not settle the residually-$p$, non-collapsing form of
Question~\ref{q2}. The full profinite Question~\ref{q1} also remains
open. Its bounded-generation condition requires uniform control over all
finite quotients, as in Lemma~\ref{lem:uniform}.
\end{remark}

\end{document}